\documentclass{amsart}

\usepackage{amsmath}
\usepackage{amssymb}
\usepackage{amsthm}
\usepackage{url}

\usepackage{hyperref}

\newtheorem{theorem}{Theorem}
\newtheorem{lemma}[theorem]{Lemma}

\newtheorem{claim}[theorem]{Claim}

\theoremstyle{definition}
\newtheorem{definition}[theorem]{Definition}

\newtheorem{question}[theorem]{Question}
\newcounter{cases}[theorem]
\newtheorem{case}[cases]{Case}

\theoremstyle{remark}
\newtheorem{remark}[theorem]{Remark}

\DeclareMathOperator{\SAME}{Same}
\DeclareMathOperator{\DIFFERENT}{Different}
\newcommand{\at}{\text{at }}

\newcommand{\G}{\mathcal{G}}
\newcommand{\N}{\mathcal{N}}
\renewcommand{\H}{\mathcal{H}}
\newcommand{\R}{\mathcal{R}}
\newcommand{\X}{\mathcal{X}}
\newcommand{\Y}{\mathcal{Y}}

\newcommand{\V}{\mathcal{V}}
\newcommand{\W}{\mathcal{W}}
\newcommand{\Z}{\mathcal{Z}}
\newcommand{\A}{\mathcal{A}}
\newcommand{\B}{\mathcal{B}}
\newcommand{\F}{\mathcal{F}}

\begin{document}

\title{Left-orderable Computable Groups}

\author[Matthew Harrison-Trainor]{Matthew Harrison-Trainor}
\address{Group in Logic and the Methodology of Science\\
University of California, Berkeley\\
 USA}
\email{matthew.h-t@berkeley.edu}
\urladdr{\href{http://www.math.berkeley.edu/~mattht/index.html}{www.math.berkeley.edu/$\sim$mattht}}

\thanks{The author was partially supported by the Berkeley Fellowship and NSERC grant PGSD3-454386-2014. The author would like to thank Antonio Montalb\'an for reading and commenting on a draft of this article.}

\begin{abstract}
Downey and Kurtz asked whether every orderable computable group is classically isomorphic to a group with a computable ordering. By an order on a group, one might mean either a left-order or a bi-order. We answer their question for left-orderable groups by showing that there is a computable left-orderable group which is not classically isomorphic to a computable group with a computable left-order. The case of bi-orderable groups is left open.
\end{abstract}

\maketitle

\section{Introduction}

A left-ordered group is a group $\G$ together with a linear order $\leq$ such that if $a \leq b$, then $c a \leq c b$. $\G$ is right-ordered if instead whenever $a \leq b$, $a c \leq b c$, and bi-ordered if $\leq$ is both a left-order and a right-order. A group which admits a left-ordering is called left-orderable, and similarly for right- and bi-orderings. A group is left-orderable if and only if it is right-orderable. Some examples of bi-orderable groups include torsion-free abelian groups and free groups \cite{Shimbireva47,Vinogradov49,Bergman90}. The group $\langle x,y : x^{-1} y x = y^{-1} \rangle$ is left-orderable but not bi-orderable. For a reference on orderable groups, see \cite{KopytovMedvedev96}.

In this paper, we will consider left-orderable computable groups. A computable group is a group with domain $\omega$ whose group operation is given by a computable function $\omega \times \omega \to \omega$. Downey and Kurtz \cite{DowneyKurtz} showed that a computable group, even a computable abelian group, which is orderable need not have a computable order. If a computable group does admit a computable order, we say that it is computably orderable. Of course, by the low basis theorem, every orderable computable group has a low ordering.

For an abelian group, any left-ordering (or right-ordering) is a bi-ordering. An abelian group is orderable if and only if it is torsion-free. Given a computable torsion-free abelian group $\G$, Dobritsa \cite{Dobritsa} showed that there is another computable group $\H$, which is classically isomorphic to $\G$, which has a computable $\mathbb{Z}$-basis. Note that $\H$ need not be computably isomorphic to $\G$. Solomon \cite{Solomon} noted that a $\mathbb{Z}$-basis for a torsion-free abelian group computes an ordering of that group. Hence every orderable computable abelian group is classically isomorphic to a computably orderable group.

Downey and Kurtz asked whether this is the case even for non-abelian groups:
\begin{question}[Downey and Kurtz \cite{DowneyRemmel}]
Is every orderable computable group classically isomorphic to a computably orderable group?
\end{question}
\noindent If one takes ``orderable'' to mean ``left-orderable'' then we give a negative answer to this question. (We leave open the question for bi-orderable groups.)
\begin{theorem}\label{thm:main}
There is a computable left-orderable group which has no presentation with a computable left-ordering.
\end{theorem}
Our strategy is to build a group \[ \G = \N \rtimes \H / \R \] and code information into the finite orbits of certain elements of $\N$ under inner automorphisms given by conjugating by elements of $\H / \R$. This strategy cannot work to build a bi-orderable group, as in a bi-orderable group there is no generalized torsion---i.e., no product of conjugates of a single element can be equal to the identity---and hence no inner automorphism has a non-trivial finite orbit. We leave open the case of bi-orderable groups.

\section{Notation}

We will use caligraphic letter such as $\G$, $\N$, and $\H$ to denote groups. For free groups, we will use upper case latin letters such as $A$, $B$, $C$, $U$, $V$, and $W$ to denote words, while using lower case letters such as $a$, $b$, and $c$ to denote letter variables. We use $\varepsilon$ for the empty word, $0$ for the identity element of abelian groups, and $1$ for the identity element of non-abelian groups (except for free groups, where we use $\varepsilon$).

\section{The Construction}

Fix $\psi$ a partial computable function which we will specify later (see Definition \ref{def:psi}). Let $p_i$, $q_i$, and $r_i$ be a partition of the odd primes into three lists.\footnote{We use the fact that $2$ does not appear in these lists in Lemma \ref{lem:powers}.} Let $\H$ be the free abelian group on $\alpha_i$, $\beta_i$, and $\gamma_i$ for $i \in \omega$. We write $\H$ additively. Let $\R$ be the set of relations
\[ \R = \{ \R_{i,t} : \psi_{\at t}(i) \downarrow \} \]
where
\[ \R_{i,t} = 
\begin{cases}
p_i^t \alpha_i = q_i^t \beta_i & \text{ if $\psi_{\at t}(i) = 0$} \\
p_i^t \alpha_i = -q_i^t \beta_i & \text{ if $\psi_{\at t}(i) = 1$} \\
\end{cases}.\]
By $\psi_{\at t}(i) = 0$, we mean that the  computation $\psi(i)$ has converged exactly at stage $t$ (but not before) and equals zero.

The idea is that these relations force, for any ordering $\leq$ on $\H / \R$, that if $\psi(i) = 0$ then $\alpha_i > 0 \Longleftrightarrow \beta_i > 0$ (and if $\psi(i) = 1$ then $\alpha_i > 0 \Longleftrightarrow \beta_i < 0$). The strategy is, in a very general sense, to use $\psi$ to diagonalize against computable orderings of $\H / \R$. The semidirect product will add enough structure to allow us to find $\alpha_i$ and $\beta_i$ within a computable copy of $\G$. (One cannot find $\alpha_i$ and $\beta_i$ within a copy of $\H/\R$, since $\H/\R$ is a torsion-free abelian group.) Note that
\[ \H / \R = \left( \bigoplus_i \langle \alpha_i, \beta_i \rangle / \R_{i} \right) \oplus \left(\bigoplus \langle \gamma_i \rangle\right) \]
where $\R_i = \R_{i,t}$ if $\psi_{\at t}(i) \downarrow$ for some $t$, or no relation otherwise.
Define
\begin{align*}
\V_i &= \R \cup \{p_i \alpha_i = 0 \} & \W_i &= \R \cup \{q_i \beta_i = 0 \} & \X_i &= \R \cup \{r_i \gamma_i = 0 \} \\
\Y_i &= \R \cup \{\alpha_i = \gamma_i \} & \Z_i &= \R \cup \{\beta_i = \gamma_i \}.
\end{align*}
Let $\N$ be the free (non-abelian) group on the letters
\begin{align*}
\{ u_i : i \in \omega \} \cup \{ v_{i,g} : g \in \H / \V_i, i \in \omega \} \cup \{ w_{i,g} : g \in \H / \W_i, i \in \omega \} \\
\cup \{ x_{i,g} : g \in \H / \X_i, i \in \omega \} \cup \{ y_{i,g} : g \in \H / \Y_i, i \in \omega\} \cup \{ z_{i,g} : g \in \H / \Z_i, i \in \omega \} .
\end{align*}
Let $\G = \N \rtimes (\H / \R)$, with $g \in \H / \R$ acting on $\N$ via the automorphism $\varphi_g$ as follows:
\begin{align*}
\varphi_g(u_i) &= u_i & \varphi_g(v_{i,h}) &= v_{i,\bar{g} + h} & \varphi_g(w_{i,h}) &= w_{i,\bar{g} + h} \\
\varphi_g(x_{i,h}) &= x_{i,\bar{g} + h} & \varphi_g(y_{i,h}) &= y_{i,\bar{g} + h} & \varphi_g(z_{i,h}) &= z_{i,\bar{g} + h}.
\end{align*}
Here, $\bar{g}$ is the image of $g$ under the quotient map $\H / \R \to \H / \V_i$ (or $\H / \W_i$, $\H / \X_i$, etc.). Recall that the semidirect product $\G = \N \rtimes (\H / \R)$ is the group with underlying set $\N \times (\H/\R)$ with group operation
\[ (n,g)(m,h) = (n \varphi_g(m), g + h).\]
Note that $\varphi_g$ permutes the letters of $\N$, and so given a word $A \in \N$, $\varphi_g(A)$ is a word of the same length as $A$. We write $\G$ multiplicatively.

\begin{lemma}
$\H / \R$ has a computable presentation.
\end{lemma}
\begin{proof}
It suffices to show that we can decide whether or not a relation of the form
\[ \sum_{i=1}^k \ell_i \alpha_i + \sum_{i=1}^k m_i \beta_i + \sum_{i=1}^k n_i \gamma_i = 0 \]
holds. This sum is equal to zero if and only if each $n_i = 0$ and for each $i$ we have $\ell_i \alpha_i + m_i \beta_i = 0$. So it suffices to decide, for a given $\ell$ and $m$ in $\mathbb{Z}$, whether $\ell \alpha_i = m \beta_i$.

Looking at $\R$, $\ell \alpha_i = m \beta_i$ if and only if either
\begin{enumerate}
	\item for some $t$, $\psi_{\at t}(i) = 0$ and there is $s \in \mathbb{Z}$ such that $\ell = s p_i^t$ and $m = s q_i^t$ or
	\item for some $t$, $\psi_{\at t}(i) = 1$ and there is $s \in \mathbb{Z}$ such that $\ell = s p_i^t$ and $m = - s q_i^t$.
\end{enumerate}
If $t > |\ell|$ or $t > |m|$ then neither of these can hold. So we just need to check, for each $t \leq |\ell|,|m|$, whether $\psi_{\at t}(i)$ converges.
\end{proof}

\begin{lemma}
$\G$ has a computable presentation.
\end{lemma}
\begin{proof}
We just need to check that $\H / \V_i$, $\H / \W_i$, and so on have computable presentations. We will see that the embeddings of the computable presentation (from the previous lemma) of $\H/\R$ into these presentations are computable. Then the action $\varphi$ of $\H / \R$ on $\N$ is computable. We can construct a computable presentation of $\G$ as the semidirect product $\N \rtimes (\H / \R)$ under this computable action.

We need to decide whether in $\H / \V_i$ we have a relation
\[ \sum_{j=1}^k \ell_j \alpha_j + \sum_{j=1}^k m_j \beta_j + \sum_{j=1}^k n_j \gamma_j = 0. \]
It suffices to decide, for a given $j$, whether
\[ \ell \alpha_j + m \beta_j + n \gamma_j = 0. \]
If $j \neq i$, this is just as in the previous lemma. Otherwise, this holds if and only if $p_i$ divides $\ell$, $q^t$ divides $m$ for some $t$ with $\psi_{\at t}(i) \downarrow$, and $n = 0$. As before, we can check this computably.

The other cases---for $\H/\W_i$, $\H/\X_i$, and so on---are similar.
\end{proof}

\begin{lemma}
$\H / \R$ is a torsion-free abelian group.
\end{lemma}
\begin{proof}
$\H/\R$ is abelian as $\H$ was abelian. Recall that
\[ \H / \R = \left(\bigoplus_i \langle \alpha_i, \beta_i \rangle / \R_{i} \right) \oplus \left(\bigoplus_i \langle \gamma_i \rangle \right) \]
where $\R_i = \R_{i,t}$ if $\psi_{\at t}(i) \downarrow$ for some $t$, or no relation otherwise.
So it suffices to show that $\langle \alpha_i, \beta_i \rangle / \R_{i}$ is torsion-free. If $\R_i$ is no relation, then this is obvious. So now suppose that $\psi_{\at t}(i) = 0$ and that
\[ k (m \alpha_i + n \beta_i) = \ell (p_i^t \alpha_i - q_i^t \beta_i) \]
in $\langle \alpha_i, \beta_i \rangle$. Since $\H$ is torsion-free, we may assume that $\gcd(k,\ell)=1$. Then $k m = \ell p_i^t$ and $k n = -\ell q_i^t$. So we must have $k= \pm 1$, in which case $m \alpha_i + n \beta_i$ is already zero in $\langle \alpha_i, \beta_i \rangle / \R_{i}$. Thus $\langle \alpha_i, \beta_i \rangle / \R_{i}$ is torsion-free. The case where $\psi_{\at t}(i) = 1$ is similar.
\end{proof}

\begin{lemma}
$\G$ is left-orderable.
\end{lemma}
\begin{proof}
Since $\H / \R$ is a torsion-free abelian group, it is bi-orderable. $\N$ is bi-orderable as it is a free group. Then by the following claim, $\G$ is left-orderable (see Theorem 1.6.2 of \cite{KopytovMedvedev96}).

\begin{claim}
Let $\A \rtimes \B$ be a semi-direct product of left-orderable groups. Then $\A \rtimes \B$ is left-orderable.
\end{claim}
\begin{proof}
Let $\varphi$ be the action of $\B$ on $\A$. Let $\leq_\A$ and $\leq_\B$ be left-orderings on $\A$ and $\B$ respectively. Define $\leq$ on $\A \rtimes \B$ as follows: $(a,b) \leq (a',b')$ if $b <_\B b'$ or $b = b'$ and $\varphi_{b^{-1}}(a) \leq_\A \varphi_{b^{-1}}(a')$. This is clearly reflexive and symmetric. We must show that it is transitive and a left-ordering.

Suppose that  $(a,b) \leq (a',b') \leq (a'',b'')$. Then $b \leq_\B b' \leq_\B b''$. If $b <_\B b''$, then $(a,b) \leq (a'',b'')$, so suppose that $b = b' = b''$. Then
\[ \varphi_{b^{-1}}(a) \leq_\A \varphi_{b^{-1}}(a') = \varphi_{{b'}^{-1}}(a') \leq_\A \varphi_{{b'}^{-1}}(a'') = \varphi_{b^{-1}}(a'').\]
So $\varphi_{b^{-1}}(a) \leq_\A \varphi_{b^{-1}}(a'')$ and so $(a,b) \leq (a'',b'')$. Thus $\leq$ is transitive.

Given $(a,b) \leq (a',b')$ we must show that $(a'',b'')(a,b) \leq (a'',b'')(a',b')$. We have that
\[ (a'',b'')(a,b) = (a'' \varphi_{b''}(a),b'' b) \text{ and } (a'',b'')(a',b') = (a'' \varphi_{b''}(a'),b'' b').\]
If $b <_\B b'$, then $b'' b <_\B b'' b'$, and so $(a'',b'')(a,b) \leq (a'',b'')(a',b')$. Otherwise, if $b = b'$ and $\varphi_{b^{-1}}(a) \leq_\A \varphi_{b^{-1}}(a')$, then $b'' b = b'' b'$ and 
\begin{align*}
\varphi_{(b'' b)^{-1}}(a'' \varphi_{b''}(a)) &= \varphi_{(b'' b)^{-1}}(a'') \varphi_{b^{-1}}(a) \\
& \leq_\A \varphi_{(b'' b)^{-1}}(a'') \varphi_{b^{-1}}(a')\\
& = \varphi_{(b'' b)^{-1}}(a'' \varphi_{b''}(a')).
\end{align*}
So $(a'',b'')(a,b) \leq (a'',b'')(a',b')$.
\end{proof}\let\qed\relax\end{proof}

Note that if $\leq$ is any left-ordering on $\G$, if $\psi_{\at t}(i) = 0$ then $(\varepsilon,\alpha_i) > 1$ if and only if $(\varepsilon,\beta_i) > 1$. On the other hand, if $\psi_{\at t}(i) = 1$ then $(\varepsilon,\alpha_i) > 1$ if and only if $(\varepsilon,\beta_i) < 1$. Later, in Definition \ref{def:same}, we will define existential formulas $\SAME(i)$ and $\DIFFERENT(i)$ (with no parameters) in the language of ordered groups. We would like to have that for any left-ordering $\leq$ on $\G$, $(\G,\leq) \models \SAME(i)$ if and only if $(\varepsilon,\alpha_i) > 1 \Longleftrightarrow (\varepsilon,\beta_i) < 1$, and $(\G,\leq) \models \DIFFERENT(i)$ if and only if $(\varepsilon,\alpha_i) > 1 \Longleftrightarrow (\varepsilon,\beta_i) < 1$. We will not quite get this for every ordering $\leq$, but this will be true for those against which we want to diagonalize (see Lemma \ref{lem:same-diff}).

\begin{definition}\label{def:psi}
Fix a list $(\F_i,\leq_i)_{i \in \omega}$ of the (partial) computable structures in the language of ordered groups. Let $\psi$ be a partial computable function with $\psi(i) = 0$ if $(\F_i,\leq_i) \models \DIFFERENT(i)$ and $\psi(i) = 1$ if $(\F_i,\leq_i) \models \SAME(i)$. It is possible, a priori, that we have both $(\F_i,\leq_i) \models \SAME(i)$ and $(\F_i,\leq_i) \models \DIFFERENT(i)$; in this case, let $\psi(i)$ be defined according to whichever existential formula we find to be true first.
\end{definition}

In fact, we will discover from the following lemma that we cannot have both $(\F_i,\leq_i) \models \SAME(i)$ and $(\F_i,\leq_i) \models \DIFFERENT(i)$.

\begin{lemma}\label{lem:same-diff}
Fix $i$. Suppose that $\F_i$ is isomorphic to $\G$ and $\leq_i$ is a computable left-ordering of $\F_i$. Let $\leq$ be an ordering on $\G$ such that $(\G,\leq) \cong (\F_i,\leq_i)$. Then:
\begin{enumerate}
	\item $(\G,\leq) \models \SAME(i)$ if and only if $(\varepsilon,\alpha_i) > 1 \Longleftrightarrow (\varepsilon,\beta_i) > 1$.
	\item $(\G,\leq) \models \DIFFERENT(i)$ if and only if $(\varepsilon,\alpha_i) > 1 \Longleftrightarrow (\varepsilon,\beta_i) < 1$.
\end{enumerate}
\end{lemma}

This lemma will be proved later. We will now show how to use Lemma \ref{lem:same-diff} to complete proof. 

\begin{lemma}
$\G$ has no computable presentation with a computable ordering.
\end{lemma}
\begin{proof}
Let $i$ be an index for $(\F_i,\leq_i)$ a computable presentation of $\G$ with a computable left-ordering. Let $\leq$ be an ordering on $\G$ such that $(\G,\leq) \cong (\F_i,\leq_i)$. Now by Lemma \ref{lem:same-diff} either $(\G,\leq) \models \SAME(i)$ or $(\G,\leq) \models \DIFFERENT(i)$ (but not both). Suppose first that $(\G,\leq) \models \SAME(i)$. So $(\F_i,\leq_i) \models \SAME(i)$. By definition, $\psi(i) = 1$, say $\psi_{\at t}(i) = 1$. Then, in $\H/\R$, $p_i^t \alpha_i = - q_i^t \beta_i$. So $(\varepsilon,\alpha_i) > 1$ if and only if $(\varepsilon,\beta_i) < 1$, contradicting Lemma \ref{lem:same-diff} and the assumption that $(\G,\leq) \models \SAME(i)$. The case of $(\G,\leq) \models \DIFFERENT(i)$ is similar. Thus $\G$ has no computable copy with a computable left-ordering.
\end{proof}

All that remains to prove Theorem \ref{thm:main} is to define $\SAME(i)$ and $\DIFFERENT(i)$ and to prove Lemma \ref{lem:same-diff}.

\section{\texorpdfstring{$\SAME(i)$, $\DIFFERENT(i)$, and the Proof of Lemma \ref{lem:same-diff}}{The Proof of Lemma \ref{lem:same-diff}}}

To define $\SAME(i)$, we would like to come up with an existential formula which says that $(\varepsilon,\alpha_i) > 1 \Longleftrightarrow (\varepsilon,\beta_i) > 1$. A first attempt might be to try to find an existential formula defining $(\varepsilon,\alpha_i)$ and an existential formula defining $(\varepsilon,\beta_i)$. This cannot be done, but it will be helpful to think about how we might try to do this.

We will consider the problem of recognizing $\alpha_i$ and $\beta_i$ inside of $\H / \R$ by their actions on $\N$. Note that $\alpha_i$ has the property that $\varphi_{\alpha_i}(v_{i,0}) = v_{i,\alpha_i} \neq 0$, but $\varphi_{p_i \alpha_i}(v_{i,0}) = v_{i,0}$. So $\alpha_i$ acts with order $p_i$ on some element of $\N$. In fact, it is not hard to see that the only elements which act with order $p_i$ on an element of $\N$ are the multiples $n \alpha_i$ of $\alpha_i$ where $p_i \nmid n$. (Note that if $\alpha_i$ acts with order $p_i$ on a word in $\N$, then it either fixes or acts with order $p_i$ on each letter in that word, and it acts with order $p_i$ on at least one letter.)

One difficulty we have is that $\H / \R$ and $\N$ are not existentially definable inside of $\G$. The problem is that if some element of $\G$ satisfies a certain existential formula, then every conjugate of $\G$ does as well. So it is only possible to define subsets of $\G$ which are closed under conjugation. Given $S \subseteq \G$, let $S^\G$ be the set of all conjugates of $S$ by elements of $\G$.

In this section, we will take for granted the following lemma about existential definability in $\G$. It will be proved in the following section. The lemma says that we can find $\H/\R$ inside of $\G$, up to conjugation, by an existential formula.

\begin{lemma}\label{cor:H/R def}
$(\H/\R)^\G$ is $\exists$-definable within $\G$ without parameters.
\end{lemma}

The different conjugates of $\H/\R$ cannot be distinguished from each other. Instead, we will try to always work inside a single conjugate of $\H/\R$. The following lemma tells us when we can do this.

\begin{lemma}\label{lem:H/R commute}
Suppose that $r,s \in (\H/\R)^\G$ and $r s \in (\H/\R)^\G$. Then there is $A \in \N$ and $g,h \in \H/\R$ such that
\[ r = (A,0)(\varepsilon,g)(A^{-1},0) \]
and
\[ s = (A,0)(\varepsilon,h)(A^{-1},0). \]
Thus $r$ and $s$ commute.
\end{lemma}

The following remarks will be helpful not only here, but throughout the rest of the paper. They can all be checked by an easy computation.
\begin{remark}\label{rem:H/R form}
If $r \in (\H/\R)^\G$, then for some $A \in \N$ and $f \in \H/\R$ we can write $r$ in the form
\[ r = (A,0)(\varepsilon,f)(A^{-1},0). \]
\end{remark}

\begin{remark}\label{rem:H/R^G}
Let $r = (A,f)$ be an element of $(\H/\R)^\G$. If $K \subseteq \H/\R$, then $r \in K^\G$ if and only if $f \in K$.
\end{remark}

\begin{remark}\label{rem:move-through}
If $\varphi_g(B) = B$, then
\[ (AB,0)(\varepsilon,g)(AB,0)^{-1} = (A,0)(\varepsilon,g)(A,0)^{-1}.\]
\end{remark}

\begin{proof}[Proof of Lemma \ref{lem:H/R commute}]
Using Remark \ref{rem:H/R form}, let
\begin{align*}
r &= (A,0)(\varepsilon,g)(A^{-1},0) & s &= (B,0)(\varepsilon,h)(B^{-1},0) \\ r s &= (C,0)(\varepsilon,g+h)(C^{-1},0).
\end{align*}
By conjugating $r$ and $s$ by some further element of $\G$ (and noting that the conclusion of the lemma is invariant under conjugation), we may assume that $A^{-1} B$ is a reduced word, that is, that $A$ and $B$ have no common non-trivial initial segment. Using Remark \ref{rem:move-through}, we may assume that $A \varphi_g({A}^{-1})$, $B \varphi_h({B}^{-1})$, and $C \varphi_{g+h}(C^{-1})$ are reduced words. Indeed, if, for example, $A \varphi_g({A}^{-1})$ was not a reduced word, then we could write $A = A' B$ where $B$ is a word which is fixed by $\varphi_g$, and such that $A' \varphi_g({A'}^{-1})$ is a reduced word. Then, by Remark \ref{rem:move-through}, 
\[ (A,0)(\varepsilon,g)(A,0)^{-1} = (A'B,0)(\varepsilon,g)(A'B,0)^{-1} = (A',0)(\varepsilon,g)(A',0)^{-1}.\]
So we may replace $A$ by $A'$.

We have
\[ (A,0)(\varepsilon,g)(A^{-1},0)(B,0)(\varepsilon,h)(B^{-1},0) = (C,0)(\varepsilon,g+h)(C^{-1},0).\]
Multiplying out the first coordinates, we get
\[ A \varphi_g({A}^{-1}) \varphi_{g}({B}) \varphi_{g+h}({B}^{-1})= C \varphi_{g+h}(C^{-1}). \]
By the assumptions we made above, both sides are reduced words. $A$ is an initial segment of the left hand side, so it must be an initial segment of the right hand side, and hence an initial segment of $C$. On the other hand, taking inverses of both sides, we get
\[ \varphi_{g+h}({B}) \varphi_g(B^{-1}) \varphi_g({A}) A^{-1} = \varphi_{g+h}(C) C^{-1}. \]
Once again both sides are reduced words, and $\varphi_{g+h}(B)$ is an initial segment of the left hand side, and hence of $\varphi_{g+h}(C)$. But then $B$ is an initial segment of $C$. So it must be that $A$ is an initial segment of $B$ or vice versa. This contradicts one of our initial assumptions unless $A$ or $B$ (or both) is the trivial word. Suppose it was $A$ (the case of $B$ is similar). Then
\[ \varphi_{g}({B}) \varphi_{g+h}({B}^{-1})= C \varphi_{g+h}(C^{-1}) \]
and both sides are reduced words. Then we get that $C = B$ and $C = \varphi_g(B)$. So
\[ r = (\varepsilon,g) = (B,0)(\varepsilon,g)(B,0)^{-1} \]
by Remark \ref{rem:move-through}.
\end{proof}

Above, we noted that the set $\{ n \alpha_i : p_i \nmid n \}$ is the set of elements of $\H/\R$ which act with order $p_i$ on an element of $\N$. Our next goal is to show that if we close under conjugation, then this set (and a few other similar sets) are definable. The key is the following remark which follows easily from Lemma \ref{lem:H/R commute}.

\begin{remark}\label{lem:fix}
Fix $r,s_1,s_2 \in (\H/\R)^\G$. Suppose that $r s_1 \in (\H/\R)^\G$ and $r s_2 \in (\H/\R)^\G$ but $s_1$ and $s_2$ do not commute. By Lemma \ref{lem:H/R commute} we can write
\begin{align*}
r &= (A,0)(\varepsilon,f)(A^{-1},0) = (B,0)(\varepsilon,f)(B^{-1},0) \\
s_1 &= (A,0)(\varepsilon,g)(A^{-1},0) \\
s_2 &= (B,0)(\varepsilon,h)(B^{-1},0).
\end{align*}
Then there is some element of $\N$ which is fixed by $\varphi_f$ but which is not fixed by $\varphi_g$. 

Indeed, since $(A,0)(\varepsilon,f)(A^{-1},0) = (B,0)(\varepsilon,f)(B^{-1},0)$, we see that
\[ B^{-1} A = \varphi_f(B^{-1} A). \]
Suppose for the sake of contradiction that $\varphi_g$ also fixes $B^{-1} A$.  Then
\[ s_1 = (A,0)(A^{-1}B,0)(\varepsilon,g)(B^{-1}A,0)(A^{-1},0) = (B,0)(\varepsilon,g)(B^{-1},0). \]
So $s_1$ and $s_2$ would commute. This is a contradiction. So there is some element of $\N$ which is fixed by $\varphi_f$ but which is not fixed by $\varphi_g$.
\end{remark}

\begin{lemma}
There are $\exists$-formulas which express each of the following statements about an element $a$ in $\G$:
\begin{enumerate}
	\item $a \in \{ n \alpha_i : p_i \nmid n \}^\G$.
	\item $a \in \{ n \beta_i : q_i \nmid n \}^\G$.
	\item $a \in \{ n \gamma_i : r_i \nmid n \}^\G$.
	\item $a \in \{ n (\alpha_i - \gamma_i) : p_i,r_i \nmid n \}^\G$.
	\item $a \in \{ n (\beta_i - \gamma_i) : q_i,r_i \nmid n \}^\G$.
\end{enumerate}
\end{lemma}
\begin{proof}
For (1), we claim that $a \in \{ n \alpha_i : p_i \nmid n \}^\G$ if and only if $a \in (\H/\R)^\G$ and there is $b \in (\H/\R)^\G$ such that $a^{p_i} b \in (\H/\R)^\G$ but $a$ and $b$ do not commute. This is expressed by an $\exists$-formula by Lemma \ref{cor:H/R def}.

Suppose that $a$ satisfies this $\exists$-formula, as witnessed by $b$. Let $a = (A,f)$ and $b = (B,g)$. Then by Remark \ref{lem:fix} (taking $r = a^{p_i}$, $s_1 = a$, and $s_2 = b$), there is an element of $\N$ which is fixed by $\varphi_{p_i f}$ but not by $\varphi_{f}$. Thus we see that $p_i \bar{f} = 0$ but $\bar{f} \neq 0$ in $\H / \V_i$, and $f = n \alpha_i$ for some $n$ with $p_i \nmid n$. (It must be in $\H / \V_i$, because this cannot happen in any of $\H / \V_j$ for $j \neq i$, or $\H / \W_j$, $\H / \X_j$, $\H / \Y_j$, or $\H / \Z_j$.) Thus by Remark \ref{rem:H/R^G}, $a \in \{ n \alpha_i : p_i \nmid n \}^\G$.

On the other hand, suppose that $a \in \{ n \alpha_i : p_i \nmid n \}^\G$. Write
\[ a = (A,0)(\varepsilon,n \alpha_i)(A^{-1},0).\]
with $p_i$ not dividing $n$. Then let $b = (A v_{i,0},0)(\varepsilon,n \alpha_i)((A v_{i,0})^{-1},0)$. By Remark \ref{rem:move-through}, since $\varphi_{n p_i \alpha_i}(v_{i,0}) = v_{i,0}$,  we have
\[ a^{p_i} = (A,0)(\varepsilon,n p_i \alpha_i)(A^{-1},0) = (A v_{i,0},0)(\varepsilon,n p_i \alpha_i)((A v_{i,0})^{-1},0).\]
So $a^{p_i} b \in (\H/\R)^\G$. On the other hand,
\[ a b = (A \varphi_{n \alpha_i}(v_{i,0}) \varphi_{2 n \alpha_i}(v_{i,0})^{-1} \varphi_{2 n \alpha_i}(A^{-1}),2 n \alpha_i) \]
and
\[ b a = (A v_{i,0} \varphi_{n \alpha_i}(v_{i,0})^{-1} \varphi_{2 n \alpha_i}(A^{-1}),2n\alpha_i). \]
So $a$ does not commute with $b$ since $\varphi_{n \alpha_i}(v_{i,0}) = v_{i,n\alpha_i} \neq v_{i,0}$.
The proofs of (2) and (3) are similar.

For (4), we claim that $a \in \{ n (\alpha_i - \gamma_i) : p_i,r_i \nmid n \}^\G$ if and only if there are $b_1 \in \{ n \alpha_i : p_i \nmid n \}^\G$, $b_2 \in \{ n \gamma_i : r_i \nmid n \}^\G$, and $c \in (\H/\R)^\G$ such that $a = b_1 b_2^{-1}$, $a c, a b_1 \in (\H/\R)^\G$, and $c$ does not commute with $b_1$.

Suppose that there are such $b_1$, $b_2$, and $c$. We can write $b_1 = (B_1,m \alpha_i)$ with $p_i \nmid m$ and $b_2 = (B_2,n \gamma_i)$ with $r_i \nmid \gamma_i$. Thus we can write $a = b_1 b_2^{-1} = (A,m \alpha_i - n \gamma_i)$. By Remark \ref{lem:fix} (with $r = a$, $s_1 = b_1$, and $s_2 = c$), $\varphi_{m \alpha_i - n \gamma_i}$ fixes some element of $\N$ which is not fixed by $\varphi_{m \alpha_i}$. Thus, in one of $\H / \V_j$, $\H / \W_j$, $\H / \X_j$, $\H / \Y_j$, or $\H / \Z_j$ for some $j$ we have $m \bar{\alpha}_i - n \bar{\gamma}_i = 0$ but $m \bar{\alpha}_i \neq 0$. Since $p_i \nmid m$,  it must be in $\H / \Y_i$. So $n = m$. Note that $p_i$ and $r_i$ do not divide $n$.

On the other hand, suppose that $a \in \{ n (\alpha_i - \gamma_i) : p_i,r_i \nmid n \}^\G$. Then write
\[ a = (A,0)(\varepsilon,n \alpha_i - n \gamma_i)(A^{-1},0).\]
with $p_i$ and $r_i$ not dividing $n$. Let
\[ b_1 = (A,0)(\varepsilon,n \alpha_i)(A^{-1},0) \text{ and } b_2 = (A,0)(\varepsilon,n \gamma_i)(A^{-1},0)\]
and let
\[ c = (A y_{i,0},0)(\varepsilon, n \alpha_i)((A y_{i,0})^{-1},0).\]
Then $a = b_1 b_2^{-1}$. Clearly $a b_1 \in (\H/\R)^\G$. Also, since $\varphi_{n \alpha_i - n \gamma_i}(y_{i,0}) = y_{i,0}$, 
\[ ac = ca = (A y_{i,0},0)(\varepsilon, 2n \alpha_i - n \gamma_i)((A y_{i,0})^{-1},0). \]
So $ac \in (\H/\R)^\G$ and $a$ and $c$ commute.
On the other hand, $b_1$ does not commute with $c$ since $\varphi_{\ell \alpha_i}(y_{i,0}) = y_{i,\ell \alpha_i} \neq y_{i,0}$ as $p_i$ does not divide $\ell$.
\end{proof}

We will now define $\SAME(i)$ and $\DIFFERENT(i)$.

\begin{definition}\label{def:same}
$\SAME(i)$ says that there are $a$, $b$, and $c$ such that:
\begin{enumerate}
	\item $a$, $b$, $c$, and $ab$ are in $(\H/\R)^\G$,
	\item $a > 1 \Longleftrightarrow b > 1$,
	\item $a \in \{ n \alpha_i : p_i \nmid n \}^\G$,
	\item $b \in \{ n \beta_i : q_i \nmid n \}^\G$,
	\item $c \in \{ n \gamma_i : r_i \nmid n \}^\G$,
	\item $a c^{-1} \in \{ n (\alpha_i - \gamma_i) : p_i,r_i \nmid n \}^\G$.
	\item $b c^{-1} \in \{ n (\beta_i - \gamma_i) : q_i,r_i \nmid n \}^\G$.
\end{enumerate}
$\DIFFERENT(i)$ is defined in the same way as $\SAME(i)$, except that in (2) we ask that $a > 1$ if and only if $b < 1$.
\end{definition}

Suppose, for simplicity, that $a$, $b$, and $c$ are all in $\H / \R$. Then we would have that $a = (\varepsilon,\ell \alpha_i)$, $b = (\varepsilon,m \beta_i)$, and $c = (\varepsilon,n \gamma_i)$. Now $a c^{-1} = (\varepsilon,\ell \alpha_i - n \gamma_i)$ is a power of $(\varepsilon,\alpha_i - \gamma_i)$, and so $\ell = n$. Similarly, $b c^{-1} = (\varepsilon,m \beta_i - n\gamma_i)$ is a power of $(\varepsilon,\beta_i - \gamma_i)$, and so $m = n$. Thus $\ell = m$. Since $(\varepsilon,\ell \alpha_i) > 1 \Longleftrightarrow (\varepsilon,\ell \beta_i) > 1$, $(\varepsilon,\alpha_i) > 1 \Longleftrightarrow (\varepsilon,\beta_i) > 1$. Checking that this works even if $a$, $b$, and $c$ are conjugates of $\H / \R$ is the heart of Lemma \ref{lem:almost-done}.

\begin{lemma}\label{lem:almost-done}
Let $\leq$ be a left-ordering on $\G$. Then:
\begin{enumerate}
	\item If $(\varepsilon,\alpha_i) > 1 \Longleftrightarrow (\varepsilon,\beta_i) > 1$, then $(\G,\leq) \models \SAME(i)$.
	\item If $(\varepsilon,\alpha_i) > 1 \Longleftrightarrow (\varepsilon,\beta_i) < 1$, then $(\G,\leq) \models \DIFFERENT(i)$.
	\item If $\psi(i) \downarrow$, then $(\varepsilon,\alpha_i) > 1 \Longleftrightarrow (\varepsilon,\beta_i) > 1$ if and only if $(\G,\leq) \models \SAME(i)$.
	\item If $\psi(i) \downarrow$, then $(\varepsilon,\alpha_i) > 1 \Longleftrightarrow (\varepsilon,\beta_i) < 1$ if and only if $(\G,\leq) \models \DIFFERENT(i)$.
\end{enumerate}
\end{lemma}
\begin{proof}
First, for (1), suppose that $(\varepsilon,\alpha_i) > 1 \Longleftrightarrow (\varepsilon,\beta_i) > 1$. Then $(\G,\leq) \models \SAME(i)$ as witnessed by $c = (\varepsilon,\alpha_i)$, $c = (\varepsilon,\beta_i)$, and $c = (\varepsilon,\gamma_i)$. (2) is similar.

Now for (3), suppose that $(\G,\leq) \models \SAME(i)$ as witnessed by $a$, $b$, and $c$, and that $\psi(i) \downarrow$. Let $f$, $g$, and $h$ be the second coordinates of $a$, $b$, and $c$ respectively. Write $f = \ell \alpha_i$ with $p_i \nmid \ell$, $g = m \beta_i$ with $q_i \nmid m$, and $h = n \gamma_i$ with $r_i \nmid h$. Then since $f - h$ is a multiple of $\alpha_i - \gamma_i$, $\ell = n$. Similarly, $m = n$, and so $\ell = m$.

Since $ab \in (\H/\R)^\G$ and $a$ and $b$ commute, by Lemma \ref{lem:H/R commute} we can write
\[ a = (B,0) (\varepsilon,\ell \alpha_i) (B,0)^{-1} \]
and
\[ b = (B,0) (\varepsilon,\ell \beta_i) (B,0)^{-1}. \]
Now since $\psi(i) \downarrow$, in $\H / \R$ either $p_i^t \alpha_i = q_i^t \beta_i$ or $p_i^t \alpha_i = - q_i^t \beta_i$ for some $t$. In the second case, $a^{p_i^t} = b^{- q_i^t}$ which contradicts the fact that $a > 1 \Longleftrightarrow b > 1$. Thus $p_i^t \alpha_i = q_i^t \beta_i$, and so $(\varepsilon,\alpha_i) > 1 \Longleftrightarrow (\varepsilon,\beta_i) > 1$.

(4) is proved similarly.
\end{proof}

\begin{proof}[Proof of Lemma \ref{lem:same-diff}]
We will prove (1):  $(\G,\leq) \models \SAME(i)$ if and only if $(\varepsilon,\alpha_i) > 1 \Longleftrightarrow (\varepsilon,\beta_i) > 1$. The proof of (2) is similar. The right to left direction follows immediately from (1) of Lemma \ref{lem:almost-done}. For the left to right direction, suppose that $(\F_i,\leq_i) \models \SAME(i)$. Then $\psi(i) \downarrow$. Then the lemma follows from (3) of Lemma \ref{lem:almost-done}.
\end{proof}

\section{\texorpdfstring{An Existential Definition of $(\H/\R)^\G$}{An Existential Definition of H/R}}

The goal of this section is to prove Lemma \ref{cor:H/R def}, which says that $(\H/\R)^\G$ is definable within $\G$ by an existential formula. To prove this lemma, we will first have to give a detailed analysis of which elements of $\G$ commute with each other.

The first lemma is the analogue of the following well-known fact about free groups: two elements $a$ and $b$ in a free group commute if and only if there is $c$ such that $a = c^m$ and $b = c^n$ (see \cite[Proposition 2.17]{LyndonSchupp}).

\begin{lemma}\label{lem:commute}
Let $r,s \in \G$ commute. Then there are $W,V \in \N$, $x,y,z \in \H/\R$, and $k,\ell \in \mathbb{Z}$ such that
\[ r = (W,0)(V,x)^k(\varepsilon,y)(W,0)^{-1} \]
and
\[ s = (W,0)(V,x)^\ell(\varepsilon,z)(W,0)^{-1}.\]
If $k \neq 0$ then $\varphi_{z}(V) = V$, and if $\ell \neq 0$ then $\varphi_{y}(V) = V$.
\end{lemma}

It is easy to check that two such elements commute.

\begin{proof}
Suppose that $r s = s r$. Let $r = (A,g)$ and $s = (B,h)$. Then we find that
\begin{align*}
r s & = (A,g)(B,h) \\
& = (A \varphi_g(B),g + h) \\
s r &= (B,h) (A,g) \\
& = (B \varphi_h(A),g + h).
\end{align*}
So $A \varphi_g(B) = B \varphi_h(A)$ in $\N$. Write
\[ A = a_0 \cdots a_{m-1} \text{ and } B = b_0 \cdots b_{n-1} \]
as reduced words. So
\[ a_0 \cdots a_{m-1} \varphi_g(b_0) \cdots \varphi_g(b_{n-1}) = b_0 \cdots b_{n-1} \varphi_h(a_0) \cdots \varphi_h(a_{m-1}). \]
We divide into several cases.

\begin{case}
$A$ is the trivial word.
\end{case}

We must have $B = \varphi_g(B)$. Then $r = (\varepsilon,g)$ and $s = (B,h)$. Take $W = \varepsilon$, $V = B$, $x = h$, $y = g$, $z = 0$, $k = 0$, and $\ell = 1$. 

\begin{case}
$B$ is the trivial word.
\end{case}

We must have $A = \varphi_h(A)$.  Then $r = (A,g)$ and $s = (\varepsilon,h)$. Take $W = \varepsilon$, $V = A$, $x = g$, $y = 0$, $z = h$, $k = 1$, and $\ell = 0$.

\begin{case}
Neither $A$ nor $B$ is the trivial word, and both $A \varphi_g(B)$ and $B \varphi_h(A)$ are reduced words.
\end{case}

We have $A \varphi_g(B) = B \varphi_h(A)$ as reduced words. Assume without loss of generality that $|A| = m \geq n = |B|$. Then $n,m > 0$ and 
\[ a_0 \cdots a_{m-1} \varphi_g(b_0) \cdots \varphi_g(b_{n-1}) = b_0 \cdots b_{n-1} \varphi_h(a_0) \cdots \varphi_h(a_{m-1}) \]
as reduced words. So
\begin{align*}
a_i &= b_i &&\text{for $0 \leq i < n$} \\
a_{i} &= \varphi_h(a_{i-n}) &&\text{for $n \leq i < m$}\\
\varphi_g(b_i) &= \varphi_h(a_{m - n + i}) &&\text{for $0 \leq i < n$}.
\end{align*}
Let $d = \gcd(m,n)$. (This is where we use the fact that $m,n > 0$.) Let $n' = n/d$ and $m' = m/d$.

Given $p,q \geq 0$, write $i = q n - p m + r$ with $0 \leq r < d$ and assume that $0 \leq i < m$. Note that every $i$, $0 \leq i < m$, can be written in such a way. We claim that
\[ a_i = \varphi_{qh - pg}(a_{r}). \]
We argue by induction, ordering pairs $(q,p)$ lexicographically. For the base case $p = q = 0$ we note that $a_r = \varphi_0(a_r)$. Otherwise, if $n \leq i < m$, then we must have $q > 0$. By the induction hypothesis, $a_{i-n} = \varphi_{(q-1)h - pg}(a_{r})$. So
\[ a_{i} = \varphi_h(a_{i-n}) = \varphi_{qh - pg}(a_{r}). \]
If $0 \leq i < n$, and $(q,p)\neq(0,0)$, then $q > 0$ and $p > 0$. Note that $a_{m - n + i} = \varphi_{(q-1)h - (p-1)g}(a_{r})$ by the induction hypothesis and so
\[ a_i = b_i = \varphi_{h-g}(a_{m - n + i}) = \varphi_{qh - pg}(a_{r}).\]
This completes the induction.

Write $d = q n - p m$ with $p,q \geq 0$. Let $f = q h - p g$. Then each $i$, $0 \leq i < m$, can be written as $i = k d + r$ with $0 \leq r < d$, and so $a_i = \varphi_{k f}(a_{r})$.

Let $C = a_0 \cdots a_{d-1}$. Then
\[ A = C \varphi_f(C) \cdots \varphi_{(m'-1)f}(C) \]
and so
\[ r = (A,g) = (C,f)^{m'}(\varepsilon,g - m' f).\]
Since for $0 \leq i < n$, $a_i = b_i$, we have
\[ s = (B,h) = (C,f)^{n'}(\varepsilon,h - n' f).\]
This is in the desired form: take $W = \varepsilon$, $V = C$, $x = f$, $y = g-m'f$, $z = h-n'f$, $k = m'$, and $\ell = n'$.

We still have to show that $\varphi_y(V) = \varphi_z(V) = V$. Noting that
\[ (n'q - 1)n - (n'p)m = n'(qn - pm) - n = n' d - n = 0 \]
we have, for all $0 \leq r < d$,
\[ a_r = \varphi_{(n'q - 1)h - n' p g}(a_r) = \varphi_{n' f - h}(a_{r}). \]
Similarly,
\[ a_r = \varphi_{m' f - g}(a_{r}). \]
Hence $\varphi_{g - m'f}(C) = \varphi_{h-n'f}(C) = C$.

\begin{case}
Neither $A$ nor $B$ is the trivial word, and both $B^{-1} A$ and $\varphi_h(A) \varphi_g(B)^{-1}$ are reduced words.
\end{case}

Note that $B^{-1} A = \varphi_h(A) \varphi_g(B)^{-1}$. We can make a transformation to reduce this to the previous case. Let
\[ A' = B^{-1} \qquad B' = \varphi_h(A) \qquad g' = -h \qquad h' = g. \]
Then $A' \varphi_{g'}(B') = B' \varphi_{h'}(A')$ and these are reduced words. Hence by the previous case there are $C \in \N$, $f \in \H/\R$, and $m,n \in \mathbb{Z}$ such that
\[ (A',g') = (C,f)^{m} (\varepsilon,g'-mf) \]
and
\[ (B',h') = (C,f)^n (\varepsilon,h'-nf) \]
and such that $\varphi_{g'-mf}(C) = C$ and $\varphi_{h'-nf}(C) = C$.
Now 
\begin{align*}
(A,g) &= (\varepsilon,-h)(\varphi_h(A),g)(\varepsilon,h) \\
&= (\varepsilon,-h)(B',h')(\varepsilon,h) \\
&= (\varepsilon,-h) (C,f)^{n} (\varepsilon,h'-nf) (\varepsilon,h)\\
&= (\varphi_{-h}(C),f)^n (\varepsilon,g-nf).
\end{align*}
Note that $\varphi_{g-nf}(C) = \varphi_{h'-nf}(C) = C$, and so $\varphi_{g-nf}(\varphi_{-h}(C)) = \varphi_{-h}(C)$. Similarly,
\begin{align*}
(B,h) &= (\varepsilon,-h)(B^{-1},-h)^{-1}(\varepsilon,h) \\
&= (\varepsilon,-h)(A',g')^{-1}(\varepsilon,h) \\
&= (\varepsilon,-h) (\varepsilon,g'-mf)^{-1} (C,f)^{-m} (\varepsilon,h)\\
&= (\varepsilon,mf) (C,f)^{-m} (\varepsilon,h)\\
&= (\varphi_{mf}(C),f)^{-m} (\varepsilon,h + mf).
\end{align*}
Since $\varphi_{h + mf}(C) = \varphi_{g'-mf}(C) = C$, $\varphi_{mf}(C) = \varphi_{-h}(C)$. So
\[ (B,h) = (\varphi_{-h}(C),f)^{-m} (\varepsilon,h + mf).\]
This completes this case, taking $W = \varepsilon$, $V = \varphi_{-h}(C)$, $x = f$, $y = g-nf$, $z = h+mf$, $k = n$, and $\ell = -m$.

\begin{case}
$|A| = 1$, $B$ is not the trivial word, and neither $A \varphi_g(B) = B \varphi_h(A)$ nor $B^{-1} A = \varphi_h(A) \varphi_g(B^{-1})$ are reduced words.
\end{case}

Let $A = a$. Then $a^{-1} = \varphi_g(b_0)$ and $b_{n-1} = \varphi_h(a^{-1})$. Recall that $B = b_0 \cdots b_{n-1}$. From the non-reduced words $A \varphi_g(B) = B \varphi_h(A)$, we get, as reduced words,
\[ \varphi_g(b_1) \varphi_g(b_2) \cdots \varphi_g(b_{n-1}) = b_0 b_1 \cdots b_{n-2}.\]
Then, for $0 \leq i < n - 1$ we get $\varphi_g(b_{i+1}) = b_{i}$. Thus $a = \varphi_{ng + h}(a)$. Also, letting $C = b_0$,
\[ r = (\varphi_g(C)^{-1},g) = (C,-g)^{-1} .\]
and
\[ s = (C,-g)^n(\varepsilon,h+ng) \]
Note that $\varphi_{h+ng}(C) = \varphi_{h+ng}(b_0) = b_0$ since $a = \varphi_{ng+h}(a)$ and $b_0 = \varphi_{-g}(a^{-1})$.

So in this case we take $W = \varepsilon$, $V = C$, $x = g$, $y = 0$, $z = h+ng$, $k = -1$, and $\ell = n$.

\begin{case}
$|B| = 1$, $A$ is not the trivial word, and neither $A \varphi_g(B) = B \varphi_h(A)$ nor $B^{-1} A = \varphi_h(A) \varphi_g(B^{-1})$ are reduced words.
\end{case}

This case is similar to the previous case.

\begin{case}
$|A|,|B| \geq 2$ and neither $A \varphi_g(B) = B \varphi_h(A)$ nor $B^{-1} A = \varphi_h(A) \varphi_g(B^{-1})$ are reduced words.
\end{case}

We have $b_{n-1} = \varphi_h(a_0)^{-1}$ and $\varphi_h(a_{m-1}) = \varphi_g(b_{n-1})$ and so
\[ \varphi_g(a_0) = \varphi_g(a_0^{-1})^{-1} = \varphi_{g-h}(b_{n-1})^{-1} = a_{m-1}^{-1}. \]
Letting
\[ A' = a_1 \cdots a_{m-2} = a_0^{-1} A \varphi_g(a_0) \]
and
\[ B' = a_0^{-1} b_0 b_1 \cdots b_{n-2} = a_0^{-1} B \varphi_h(a_0) \]
we have
\begin{align*}
B' \varphi_h(A') \varphi_g(B')^{-1} &= B' b_{n-1} \varphi_h(a_0) \varphi_h(A') \varphi_h(a_{m-1}) \varphi_g(b_{n-1})^{-1} \varphi_g(B')^{-1} \\
&= a_0^{-1} B \varphi_h(A) \varphi_g(B)^{-1} a_{m-1}^{-1} \\
& = a_0^{-1} A a_{m-1}^{-1} \\
& = A'.
\end{align*}
So $(A',g)$ and $(B',h)$ still commute.

Note that $|A'| < |A|$ and $|B'| \leq |B|$. So we only have to repeat this finitely many times until we are in one of the other cases. Thus, for some word $D$ we get reduced words
\[ A' = D A \varphi_g(D^{-1}) \]
and
\[ B' = D B \varphi_h(D^{-1}) \]
which fall into one of the other cases. So
\[ (A',g) = (C,f)^m(\varepsilon,g-mf) \]
and
\[ (B',h) = (C,f)^n(\varepsilon,h-nf).\]
Thus
\[ r = (D A' \varphi_g(D^{-1}),g) = (D,0)(A',g)(D^{-1},0) \]
and
\[ s = (D B' \varphi_h(D^{-1}),h) = (D,0)(B',h)(D^{-1},0)\]
are in the desired form.
\end{proof}

The next lemma gives a criterion for knowing that an element $r$ is in $(\H/\R)^\G$, but it requires knowing that two particular elements $s_1$ and $s_2$ are not in $(\H/\R)^\G$. This does not seem useful yet, but in Lemma \ref{lem:guarantee} we will show that any three elements $s_1$, $s_2$, and $s_3$, such that $r$ commutes with each of them but $s_1$, $s_2$, and $s_3$ pairwise do not commute, give rise to two such elements which are not in $(\H/\R)^\G$.

\begin{lemma}\label{lem:in H/R}
Let $r, s_1, s_2 \in \G$. Suppose that $r$ commutes with $s_1$ and $s_2$, but $s_1$ and $s_2$ do not commute. If $s_1,s_2 \notin (\H/\R)^\G$, then $r \in (\H/\R)^\G$.
\end{lemma}
\begin{proof}
Suppose to the contrary that $r \notin (\H/\R)^\G$. Since $r$ and $s_1$ commute, and $r$ and $s_2$ commute, by Lemma \ref{lem:commute} we can write
\begin{align*}
r &= (A,0) (C,f_1)^{m_1} (\varepsilon,g_1) (A^{-1},0) = (B,0) (D,f_2)^{m_2} (\varepsilon,g_2) (B^{-1},0) \\
s_1 &= (A,0) (C,f_1)^{n_1} (\varepsilon,h_1) (A^{-1},0) \\
s_2 &= (B,0) (D,f_2)^{n_2} (\varepsilon,h_2) (B^{-1},0)
\end{align*}
Since $r$, $s_1$, and $s_2$ are not in $(\H/\R)^\G$, $C$ and $D$ are non-trivial and $m_1,m_2,n_1,n_2 \neq 0$. So $\varphi_{g_1}(C) = \varphi_{h_1}(C) = C$ and $\varphi_{g_2}(D) = \varphi_{h_2}(D) = D$. Moreover, we will argue that we may assume that
\[ C \varphi_{f_1}(C) \cdots \varphi_{(m_1 - 1)f_1}(C) \text{ and } D \varphi_{f_2}(D) \cdots \varphi_{(m_2 - 1)f_2}(D) \]
are reduced words. If the former is not a reduced word, then it must have length at least 2, and we can write $C = a C' \varphi_{f_1}(a^{-1})$. Then
\[ C \varphi_{f_1}(C) \cdots \varphi_{(m_1 - 1)f_1}(C) = a C' \varphi_{f_1}(C') \cdots \varphi_{(m_1 - 1)f_1}(C') \varphi_{m_1 f_1}(a^{-1}) \]
and so, since $\varphi_{g_1}$ fixes $C$ and hence $a$,
\[ r = (A a, 0)(C',f_1)^{m_1}(\varepsilon,g_1)(a^{-1} A^{-1},0).\]
Similarly,
\[ s_1 = (A a, 0)(C',f_1)^{n_1}(\varepsilon,h_1)(a^{-1} A^{-1},0).\]
So we may replace $A$ by $Aa$ and $C$ by $C'$. We can continue to do this until $C \varphi_{f_1}(C) \cdots \varphi_{(m_1 - 1)f_1}(C)$ is a reduced word. The same argument works for $D \varphi_{f_2}(D) \cdots \varphi_{(m_2 - 1)f_2}(D)$.

Rearranging the two expressions for $r$, we get
\[ (B^{-1} A,0) (C,f_1)^{m_1} (\varphi_{g_1}(A^{-1} B),g_1) = (D,f_2)^{m_2} (\varepsilon,g_2). \]
Looking at the first coordinate,
\begin{eqnarray*}
&&B^{-1} A C \varphi_{f_1}(C) \varphi_{2 f_1}(C) \cdots \varphi_{(m_1-1)f_1}(C) \varphi_{m_1 f_1 + g_1}(A^{-1}B) \\
&=& D \varphi_{f_2}(D) \varphi_{2 f_2}(D) \cdots \varphi_{(m_2 - 1) f_2}(D).
\end{eqnarray*}
We claim that we can write $B^{-1} A = E_2^{-1} E_1$ where $\varphi_{g_1}(E_1) = \varphi_{h_1}(E_1) = E_1$ and $\varphi_{g_2}(E_2) = \varphi_{h_2}(E_2) = E_2$.
Recall that 
\[ C \varphi_{f_1}(C) \varphi_{2 f_1}(C) \cdots \varphi_{(m_1-1)f_1}(C) \]
is a non-trivial reduced word. Taking a high enough power $\ell$, the length of
\[ (C \varphi_{f_1}(C) \varphi_{2 f_1}(C) \cdots \varphi_{(m_1-1)f_1}(C))^{\ell} \]
as a reduced word is more than twice the length of $B^{-1} A$. Then
\begin{eqnarray*}
&& B^{-1} A (C \varphi_{f_1}(C) \varphi_{2 f_1}(C) \cdots \varphi_{(m_1-1)f_1}(C))^{\ell} \varphi_{m_1 f_1 + g_1}(A^{-1}B) \\
&=& (D \varphi_{f_2}(D) \varphi_{2 f_2}(D) \cdots \varphi_{(m_2 - 1) f_2}(D))^{\ell}.
\end{eqnarray*}
We can write $B^{-1} A = E_2^{-1} E_1$ as a reduced word where $E_2^{-1}$ appears at the start of the right hand side when it is written as a reduced word, and $E_1$ cancels with the beginning of $(C \varphi_{f_1}(C) \varphi_{2 f_1}(C) \cdots \varphi_{(m_1-1)f_1}(C))^{\ell}$. Thus $E_1$ is fixed by $\varphi_{g_1}$ and $\varphi_{h_1}$ since they fix each letter appearing in the word $(C \varphi_{f_1}(C) \varphi_{2 f_1}(C) \cdots \varphi_{(m_1-1)f_1}(C))^{\ell}$, and $E_2$ is fixed by $\varphi_{g_2}$ and $\varphi_{h_2}$ since they fix each letter appearing in the right hand side.

Since $E_2 B^{-1} = E_1 A^{-1}$,
\begin{align*}
E_2 B^{-1} r B E_2^{-1} &= (E_1,0) (C,f_1)^{m_1} (\varepsilon,g_1) (E_1^{-1},0) \\
& = (E_2,0) (D,f_2)^{m_2} (\varepsilon,g_2) (E_2^{-1},0) \\
E_2 B^{-1} s_1 B E_2^{-1} &= (E_1,0) (C,f_1)^{n_1} (\varepsilon,h_1) (E_1^{-1},0) \\
E_2 B^{-1} s_2 B E_2^{-1} &= (E_2,0) (D,f_2)^{n_2} (\varepsilon,h_2) (E_2^{-1},0).
\end{align*}
So, applying the automorphism of $\G$ given by conjugating by $E_2 B^{-1}$ (and noting that this automorphism fixes $(\H/\R)^\G$) we may assume from the beginning that $\varphi_{g_1}(A) = \varphi_{h_1}(A) = A$ and $\varphi_{g_2}(B) = \varphi_{h_2}(B) = B$. Thus
\begin{align*}
r &= (A,0) (C,f_1)^{m_1} (A^{-1},0) (\varepsilon,g_1) = (B,0) (D,f_2)^{m_2} (B^{-1},0) (\varepsilon,g_2) \\
s_1 &= (A,0) (C,f_1)^{n_1} (A^{-1},0) (\varepsilon,h_1) \\
s_2 &= (B,0) (D,f_2)^{n_2} (B^{-1},0) (\varepsilon,h_2). 
\end{align*}

Now looking at the first coordinate, we have
\begin{eqnarray*}
&& A C \varphi_{f_1}(C) \varphi_{2 f_1}(C) \cdots \varphi_{(m_1-1)f_1}(C) \varphi_{m_1 f_1}(A)^{-1} \\
&=& B D \varphi_{f_2}(D) \varphi_{2 f_2}(D) \cdots \varphi_{(m_2 - 1) f_2}(D) \varphi_{m_2 f_2}(B)^{-1}.
\end{eqnarray*}
Our next step is to argue that we may assume that these are reduced words. Suppose that there was some cancellation, say $A = A' a$ and $C = a^{-1} C'$. Let $C^* = C' \varphi_{f_1}(a^{-1})$. Then
\begin{eqnarray*}
&& A C \varphi_{f_1}(C) \varphi_{2 f_1}(C) \cdots \varphi_{(m_1-1)f_1}(C) \varphi_{m_1 f_1}(A)^{-1} \\
&=& A' C^* \varphi_{f_1}(C^*) \varphi_{2 f_1}(C^*) \cdots \varphi_{(m_2-1)f_1}(C^*) \varphi_{m_1 f_1}(A')^{-1}.
\end{eqnarray*}
Thus
\begin{align*}
r &= (A',0) (C^*,f_1)^{m_1} (\varepsilon,g_1) (A',0)^{-1} \\
s_1 &= (A',0) (C^*,f_1)^{n_1} (\varepsilon,h_1) (A',0)^{-1}.
\end{align*}
Note that
\[ (C^*,f_1)^{m_1} = C^* \varphi_{f_1}(C^*) \varphi_{2 f_1}(C^*) \cdots \varphi_{(m_1-1)f_1}(C^*) \]
is still a reduced word. If it was not a reduced word, then we would have $m_1 > 0$, $|C^*| > 1$, and $\varphi_{f_1}(a^{-1}) = \varphi_{f_1}(a')^{-1}$, where $a'$ is the first letter of $C^*$. Thus $a' = a$ is the second letter of $C$, which together with the fact that the first letter of $C$ is $a^{-1}$ contradicts our assumption that $C$ is a reduced word. We have reduced the size of $A$, so after finitely many reductions of this form, we get
\begin{eqnarray*}
&& A C \varphi_{f_1}(C) \varphi_{2 f_1}(C) \cdots \varphi_{(m_1-1)f_1}(C) \varphi_{m_1 f_1}(A)^{-1} \\
&=& B D \varphi_{f_2}(D) \varphi_{2 f_2}(D) \cdots \varphi_{(m_2 - 1) f_2}(D) \varphi_{m_2 f_2}(B)^{-1}
\end{eqnarray*}
and that both sides are reduced words.

Now either $|A| \leq |B|$ or $|B| \leq |A|$. Without loss of generality, assume that we are in the first case. Then $A$ is an initial segment of $B$ (i.e., $B = A B'$ as a reduced word). Then by replacing $r$, $s_1$, and $s_2$ with $A^{-1} r A$, $A^{-1} s_1 A$, and $A^{-1} s_2 A$, we may assume that $A$ is trivial. To summarize the reductions we have made so far, we have
\begin{align*}
r &= (C,f_1)^{m_1} (\varepsilon,g_1) = (B,0) (D,f_2)^{m_2} (\varepsilon,g_2) (B^{-1},0) \\
s_1 &= (C,f_1)^{n_1} (\varepsilon,h_1) \\
s_2 &= (B,0) (D,f_2)^{n_2} (\varepsilon,h_2) (B^{-1},0).
\end{align*}
The automorphisms $\varphi_{g_1}$ and $\varphi_{h_1}$ fix $C$, and the automorphisms $\varphi_{g_2}$ and $\varphi_{h_2}$ fix $D$ and $B$. Both sides of
\begin{eqnarray*}
&& C \varphi_{f_1}(C) \varphi_{2 f_1}(C) \cdots \varphi_{(m_1-1)f_1}(C) \\
&=& B D \varphi_{f_2}(D) \varphi_{2 f_2}(D) \cdots \varphi_{(m_2 - 1) f_2}(D) \varphi_{m_2 f_2}(B)^{-1}
\end{eqnarray*}
are reduced words.

Now we will show that either $m_1 = 1$ or $B$ is trivial. Suppose that $B$ was non-trivial, say $B = b B'$. First note that the length of $C$ is greater than one, as otherwise $C = b$ and $\varphi_{(m_1-1) f_1}(C) = \varphi_{m_2 f_2}(b^{-1})$; but there is no $e \in \H/\R$ such that $\varphi_e(b) = b^{-1}$. Then we must have $C = b C' \varphi_{m_2 f_2 - (m_1-1) f_1}(b^{-1})$ for some $C'$. We have $m_1 f_1 + g_1 = m_2 f_2 + g_2$. Since $b$ appears both in $C$ and in $B$, it is fixed by both $\varphi_{g_1}$ and $\varphi_{g_2}$. Thus $C = b C' \varphi_{f_1}(b^{-1})$. But then if $m_1 > 1$,
\[ C \varphi_{f_1}(C) \varphi_{2 f_1}(C) \cdots \varphi_{(m_1-1)f_1}(C) \]
is not a reduced word. So we conclude that either $m_1 = 1$ or $B$ is trivial.

\begin{case}
Suppose that $m_1 = 1$. 
\end{case}

We have
\[ r = (C,f_1)(\varepsilon,g_1) = (B,0) (D,f_2)^{m_2} (\varepsilon,g_2) (B^{-1},0). \]
Also, as reduced words,
\[ C = B D \varphi_{f_2}(D) \varphi_{2 f_2}(D) \cdots \varphi_{(m_2 - 1) f_2}(D) \varphi_{m_2 f_2}(B)^{-1}. \]
Since the right hand side is a reduced word, $\varphi_{g_1}$ and $\varphi_{h_1}$ fix $B$ and $D$ since each letter in $B$ and $D$ appears in $C$. Thus
\[ s_1 = (C,f_1)^{n_1} (\varepsilon,h_1) = [(B,0)(D,f_2)^{m_2}(B^{-1},0) (\varepsilon,f_1 - m_2 f_2)]^{n_1} (\varepsilon,h_1).\]
Now $f_1 + g_1 = m_2 f_2 + g_2$. Since $\varphi_{g_1}$ and $\varphi_{g_2}$ fix $B$ and $D$, $\varphi_{f_1 - m_2 f_2}$ also fixes $B$ and $D$. Thus
\[ s_1 = (B,0)(D,f_2)^{m_2 n_1} (\varepsilon,h_1 + n_1(f_1 - m_2 f_2)) (B^{-1},0) \]
and $h_1 + n_1(f_1 - m_2 f_2)$ fixes $D$. Thus $s_1$ and $s_2$ commute. This is a contradiction.

\begin{case}
$B$ is trivial.
\end{case}

Let $|C| = k$ and $|D| = \ell$. Suppose without loss of generality that $k \geq \ell$.
Let $d_0,d_1,d_2,\ldots$ be the reduced word
\[ C \varphi_{f_1}(C) \varphi_{2 f_1}(C) \cdots \varphi_{(m_1-1)f_1}(C) = D \varphi_{f_2}(D) \varphi_{2 f_2}(D) \cdots \varphi_{(m_2-1)f_2}(D).\]
Then we have
\begin{align*}
d_i &= \varphi_{f_2}(d_{i-\ell}) &&\text{for $i \geq \ell$}\\
\varphi_{(m_1 - 1)f_1}(d_{k - \ell + i}) &= \varphi_{(m_2 - 1)f_2}(d_{i}) &&\text{for $0 \leq i < \ell$}
\end{align*}
Let $e = \gcd(k,\ell)$.

Given $p,q \geq 0$, write $i = q \ell - p k + r$ with $0 \leq r < e$ and assume that $0 \leq i < m_1 k = m_2 \ell$. Note that every $i$, $0 \leq i < m_1 k = m_2 \ell$, can be written in such a way. We claim that
\[ d_i = \varphi_{qf_2 + p [(m_1-1) f_1 - m_1 f_2]}(d_{r}). \]
We argue by induction, ordering pairs $(q,p)$ lexicographically. For the base case $p = q = 0$ we note that $d_r = \varphi_0(d_r)$. 
If $\ell \leq i$, then we must have $q > 0$. By the induction hypothesis, $d_{i-\ell} = \varphi_{(q-1)f_2 + p [(m_1-1)f_1 - m_2f_2]}(d_r)$. So
\[ d_{i} = \varphi_{f_2}(d_{i-\ell}) = \varphi_{qf_2 + p [(m_1-1) f_1 - m_2f_2]}(d_{r}). \]
If $0 \leq i < \ell$, and $(q,p)\neq(\varepsilon,0)$, then $q > 0$ and $p > 0$. Note that
\[ d_{k - \ell + i} = \varphi_{(q-1)f_2 + (p-1) [(m_1-1) f_1 - m_2f_2]}(d_{r}) = \varphi_{(q f_2 + p [(m_1-1) f_1 - m_2f_2] - [(m_1-1)f_1 - (m_2-1)f_2]}(d_{r}) \]
by the induction hypothesis and so
\[ d_i = \varphi_{(m_1 - 1)f_1 - (m_2 - 1)f_2}(d_{i+k-\ell}) = \varphi_{qf_2 + p [(m_1-1) f_1 - m_2f_2]}(c_{r}).\]
This completes the induction.

Write $e = q \ell - p k$ with $p,q \geq 0$. Let $f = qf_2 + p [(m_1-1) f_1 - m_2f_2]$. Then each $i$, $0 \leq i < k m_1$, can be written as $i = s e + r$ with $0 \leq r < d$, and so
\[ d_i = \varphi_{s f}(d_{r}). \]
Let $E = d_1 \cdots d_e$. Then
\[ C = E \varphi_f(E) \cdots \varphi_{(\frac{k}{e}-1)f}(E). \]
Similarly,
\[ D = E \varphi_f(E) \cdots \varphi_{(\frac{\ell}{e}-1)f}(E). \]
Also,
\[ \varphi_{f_1}(E) = d_k \cdots d_{k + e - 1} = \varphi_{\frac{k}{e} f}(d_0,\ldots,d_{e-1}) = \varphi_{\frac{k}{e} f}(E) \]
and
\[ \varphi_{f_2}(E) = d_\ell \cdots d_{\ell + e - 1} = \varphi_{\frac{\ell}{e} f}(d_0,\ldots,d_{e-1}) = \varphi_{\frac{\ell}{e} f}(E). \]
So $\varphi_{f_1}(C) = \varphi_{\frac{k}{e} f}(C)$ and $\varphi_{f_2}(D) = \varphi_{\frac{\ell}{e} f}(D)$.
Hence
\[ s_1 = (C,f_1)^{m_1} (\varepsilon,h_1) = (E,f)^{\frac{m_1 k}{e}} (\varepsilon,h_1 + m_1 f_1 - \frac{m_1 k}{e} f) \]
and
\[ s_2 = (D,f_2)^{m_1} (\varepsilon,h_2) = (E,f)^{\frac{m_2 \ell}{e}} (\varepsilon,h_2 + m_2 f_2 - \frac{m_2 \ell}{e} f) \]
Note that $\varphi_{h_1}$ and $\varphi_{h_2}$ both fix $E$, since they fix $C$ and $D$ respectively. Also, since $\varphi_{f_1}(E) = \varphi{\frac{k}{e}f}(E)$, $\varphi_{m_1 f_1 - \frac{m_1 k}{e} f}$ fixes $E$. Similarly, $\varphi_{m_2 f_2 - \frac{m_2 \ell}{e} f}$ fixes $E$ . So
$s_1$ and $s_2$ commute. This is a contradiction.
\end{proof}

\begin{lemma}\label{lem:powers}
Fix $r \in \G$. If $r^2 \in \H/\R$, then $r \in \H/\R$.
\end{lemma}
\begin{proof}
Write $r = (A,f)$. We will show that if $r \notin \H/\R$, i.e.\ if $A \neq \varepsilon$, then $r^2 \notin \H/\R$. Since
\[ r^2 = (A \varphi_f(A),2f) \]
we must show that $A \varphi_f(A)$ is non-trivial. Suppose that it was trivial; then the length of $A$ as a reduced word must be even. (If the length of $A$ was odd, say $A = A_1 a A_2$ with $A_1$ and $A_2$ of equal lengths, then \[ A \varphi_f(A) = A_1 a A_2 \varphi_f(A_1) \varphi_f(a) \varphi_f(A_2) = \varepsilon.\] So it must be that $\varphi_f(a) = a^{-1}$, which cannot happen for any letter $a$.) Write $A = B C$, where $B$ and $C$ are each half the length of $A$. Then since $A \varphi_f(A)$ is the trivial word, $C \varphi_f(B)$ is the trivial word; thus $C = \varphi_f(B^{-1})$. So $A = B \varphi_f(B^{-1})$, and
\[ A \varphi_f(A) = B \varphi_f(B^{-1}) \varphi_f(B) \varphi_{2f}(B^{-1}) = B \varphi_{2f}(B^{-1}). \]
Since $A \varphi_f(A)$ is the trivial word, $\varphi_{2f}(B) = B$. Since $A$ is not the trivial word, $B \neq \varphi_f(B)$. But this is impossible, as $p_i$, $q_i$, and $r_i$ were all chosen to be odd primes.
\end{proof}

The next lemma is the heart of the existential definition of $(\H/\R)^\G$. The proof is to show that under the hypotheses of the lemma, elements not in $(\H/\R)^\G$ such as in Lemma \ref{lem:in H/R} must exist.

\begin{lemma}\label{lem:guarantee}
Let $r, s_1, s_2, s_3 \in \G$. Suppose that $r$ commutes with $s_1$, $s_2$, and $s_3$, but that no two of $s_1$, $s_2$, and $s_3$ commute. Then $r \in (\H/\R)^\G$.
\end{lemma}
\begin{proof}
If at least two of $s_1$, $s_2$, and $s_3$ are not in $(\H/\R)^\G$, then this follows immediately by Lemma \ref{lem:in H/R}. Otherwise, without loss of generality suppose that $s_1$ and $s_2$ are in $(\H/\R)^\G$. By Lemma \ref{lem:H/R commute}, $s_1 s_2 \notin (\H/\R)^\G$.

Note that $r$ commutes with $s_1 s_2$ and with $s_1 (s_2)^2$. Also, $s_1 s_2$ does not commute with $s_1 (s_2)^2$, since if it did, then
\[ s_1 s_2 s_1 s_2 s_2 = s_1 s_2 s_2 s_1 s_2 \Rightarrow s_1 s_2 = s_2 s_1.\]
We claim that $s_1 (s_2)^2 \notin (\H/\R)^\G$. If $s_1 (s_2)^2$ was in $(\H/\R)^\G$, then by Lemma \ref{lem:H/R commute}, we could write
\[ s_1 = (A,0)(\varepsilon,g)(A^{-1},0) \text{ and } (s_2)^2 = (A,0)(\varepsilon,h)(A^{-1},0). \]
Then let $s_2' = (A^{-1},0) s_2 (A,0) = (C,f)$. Then $(s_2')^2 = (\varepsilon,h)$, and so by Lemma \ref{lem:powers}, $s_2' = (\varepsilon,f)$. Thus $s_2 = (A,0)(\varepsilon,f)(A^{-1},0)$. So $s_1$ and $s_2$ would commute; since we know that $s_1$ and $s_2$ do not commute, $s_1 (s_2)^2 \notin (\H/\R)^\G$.

By Lemma \ref{lem:in H/R}, with $r$, $s_1 s_2$, and $s_1 s_2^2$, we see that $r$ is in $(\H/\R)^\G$.
\end{proof}

The existential definition of $(\H/\R)^\G$ comes from the previous lemma. It remains only to show that if $r \in (\H/\R)^\G$, then the hypothesis of the previous lemma is satisfied.

\begin{proof}[Proof of Lemma \ref{cor:H/R def}]
By the previous lemma, it suffices to show that if $r \in (\H/\R)^\G$, then there are $s_1$, $s_2$, and $s_3$ such that $r$ commutes with $s_1$, $s_2$, and $s_3$, but no two of these commute with each other. If $r = (A,0)(\varepsilon,g)(A^{-1},0)$, let $s_1 = (A,0)(u_0,0)(A^{-1},0)$, $s_2 = (A,0)(u_1,0)(A^{-1},0)$, and $s_3 = (A,0)(u_2,0)(A^{-1},0)$. Then $r$ commutes with $s_1$, $s_2$, and $s_3$ since $g$ fixes $u_0$, $u_1$, and $u_2$, but no two of $s_1$, $s_2$, and $s_3$ commute with each other as $u_0$, $u_1$, and $u_2$ do not commute with each other.
\end{proof}

\bibliography{References}

\begin{thebibliography}{Dob83}

\bibitem[Ber90]{Bergman90}
G.~M. Bergman.
\newblock Ordering coproducts of groups and semigroups.
\newblock {\em J. Algebra}, 133(2):313--339, 1990.

\bibitem[DK86]{DowneyKurtz}
R.~G. Downey and S.~A. Kurtz.
\newblock Recursion theory and ordered groups.
\newblock {\em Ann. Pure Appl. Logic}, 32(2):137--151, 1986.

\bibitem[Dob83]{Dobritsa}
V.~P. Dobritsa.
\newblock Some constructivizations of abelian groups.
\newblock {\em Sibirsk. Mat. Zh.}, 24(2):18--25, 1983.

\bibitem[DR00]{DowneyRemmel}
R.~G. Downey and J.~B. Remmel.
\newblock Questions in computable algebra and combinatorics.
\newblock In {\em Computability theory and its applications ({B}oulder, {CO},
  1999)}, volume 257 of {\em Contemp. Math.}, pages 95--125. Amer. Math. Soc.,
  Providence, RI, 2000.

\bibitem[KM96]{KopytovMedvedev96}
V.~M. Kopytov and N.~Ya. Medvedev.
\newblock {\em Right-ordered groups}.
\newblock Siberian School of Algebra and Logic. Consultants Bureau, New York,
  1996.

\bibitem[LS01]{LyndonSchupp}
R.~C. Lyndon and P.~E. Schupp.
\newblock {\em Combinatorial group theory}.
\newblock Classics in Mathematics. Springer-Verlag, Berlin, 2001.
\newblock Reprint of the 1977 edition.

\bibitem[Shi47]{Shimbireva47}
H.~Shimbireva.
\newblock On the theory of partially ordered groups.
\newblock {\em Rec. Math. [Mat. Sbornik] N.S.}, 20(62):145--178, 1947.

\bibitem[Sol02]{Solomon}
R.~Solomon.
\newblock {$\Pi_1^0$} classes and orderable groups.
\newblock {\em Ann. Pure Appl. Logic}, 115(1-3):279--302, 2002.

\bibitem[Vin49]{Vinogradov49}
A.~A. Vinogradov.
\newblock On the free product of ordered groups.
\newblock {\em Mat. Sbornik N.S.}, 25(67):163--168, 1949.

\end{thebibliography}
\bibliographystyle{alpha}

\end{document}